\documentclass[12pt]{amsart}

\usepackage{amssymb}
\usepackage{graphicx}
\usepackage{enumerate}
\usepackage{multirow}
\usepackage{amsmath,color}
\usepackage{hyperref}
\usepackage{url}
\usepackage[section]{placeins}

\newtheorem{thm}{Theorem}[section]
\newtheorem{cor}[thm]{Corollary}

\newtheorem{prop}[thm]{Proposition}

\newtheorem{Question}[thm]{Question}
\numberwithin{equation}{section}
\newcommand{\seq}[1]{\langle #1\rangle}
\makeatletter
\@namedef{subjclassname@2020}{%
  \textup{2020} Mathematics Subject Classification}
\makeatother

\title{Self-complementary distance-regular Cayley graphs over abelian groups}

\author{Mojtaba Jazaeri}
\address{Department of Mathematics, Shahid Chamran University of Ahvaz, Ahvaz, Iran}
\email{M.Jazaeri@scu.ac.ir, M.Jazaeri@ipm.ir}
\begin{document}



\begin{abstract}
In this paper, we study self-complementary distance-regular Cayley graphs over abelian groups. We prove that if a regular graph is self-complementary distance-regular, then it is self-complementary strongly regular. We also deal with self-complementary strongly regular Cayley graphs over abelian groups and give an example of a self-complementary strongly regular Cayley graph over a non-elementary abelian group.
\end{abstract}

\maketitle


\section{Introduction}
A graph is called self-complementary if it is isomorphic to its complement. The structure of a self-complementary graph is beautiful and it has also application in finding lower bound in Ramsey numbers. It is well known that a self-complementary graph has diameter at most three. The classification of self-complementary regular graphs seems to be complicated and therefore the researchers have focused on some special families such as strongly regular graphs, Cayley graphs, vertex-transitive graphs, and symmetric graphs. For some works on this topic, we refer to \cite{WLLW} and its references. It turns out that the diameter of a self-complementary regular graph is two (see Proposition \ref{Self-complementary regular} below). Furthermore, a regular graph with diameter two has at least three distinct eigenvalues (with respect to the adjacency matrix) and therefore, a first step could be concentrated on self-complementary regular graphs of diameter two with three distinct eigenvalues. These graphs are indeed self-complementary strongly regular graphs. In general, the classification of self-complementary strongly regular graphs seems to be hard and there is no outstanding achievement in literature related to this family. Therefore it is better to restrict this family. In this paper, we concentrate on self-complementary strongly-regular Cayley graphs over abelian groups. In general, the study of strongly regular Cayley graphs are the same as partial difference sets (see \cite{MA}) and self-complementary strongly regular Cayley graphs come from Paley type partial difference sets. Arasu, Jungnickel, Ma and Pott stated the following questions in \cite{AJMP}.
\begin{Question} \label{Question1}
Let $G$ be an abelian group of order $4t+1$. If $4t+1$ is not a prime power, does there exist a Paley type partial difference set in the group $G$? If $4t+1$ is a prime power, does the group $G$ need to be elementary abelian?
\end{Question}
Davis in Corollary $3.1$ of \cite{Davis} answered the second part of this question by proving that there exists a Paley type partial difference set in the abelian group $\mathbb{Z}_{p^{2}} \times \mathbb{Z}_{p^{2}}$. Furthermore, Leung and Ma  in \cite{LM} generalized this by the construction of Paley type partial difference set in abelian $p$-groups with any given exponent. Moreover, Polhill  in \cite{Polhill} proved that there exist Paley type partial difference sets in the group $\mathbb{Z}_{3}^{2} \times \mathbb{Z}_{p}^{4t}$, where $t$ is a natural number and $p$ is an odd prime number that answered the first part of this question. Finally, the order of abelian groups which admit Paley type partial difference sets has been discovered in \cite{Wang}. It is proven that if an abelian group admits a Paley type partial difference set and its order is not a prime power, then its order is $n^{4}$ or $9n^{4}$, where $n>1$ is an odd integer, but, however, Paley type partial difference sets which come from self-complementary strongly regular Cayley graphs over abelian groups could be restricted than this. Therefore we state a similar question as follows.
\begin{Question} \label{Question2}
Let $G$ be an abelian group of order $4t+1$. If $4t+1$ is not a prime power, does there exist a self-complementary strongly regular graph in the group $G$? If $4t+1$ is a prime power, does the group $G$ need to be elementary abelian?
\end{Question}
In this paper, we deal with self-complementary strongly regular Cayley graphs over abelian groups, and partially answer the second part of this question by giving an example of a self-complementary strongly regular Cayley graph over the abelian group $\mathbb{Z}_{9} \times \mathbb{Z}_{9}$.
\section{Preliminaries}
In this paper, all graphs are undirected and simple, i.e., there are no loops or multiple edges. Moreover, we consider the eigenvalues of the adjacency matrix of a graph. A connected graph $\Gamma$ is called distance-regular with diameter $d$ and intersection array
\begin{equation*}
\{b_{0},b_{1},\ldots,b_{d-1};c_{1},c_{2},\ldots,c_{d}\}
\end{equation*}
whenever for each pair of vertices $x$ and $y$ at distance $i$, where $0 \leq i \leq d$, the number of neighbors of $x$ at distance $i+1$ and $i-1$ from $y$ are constant numbers $b_{i}$ and $c_{i}$, respectively. This implies that a distance-regular graph is regular with valency $b_{0}=k$ and the number of neighbors of $x$ at distance $i$ from $y$ is the constant number $k-b_{i}-c_{i}$ which is denoted by $a_{i}$. A distance-regular graph with diameter $d=2$ is a strongly regular graph. It is well known that a strongly regular graph with $n$ vertices and degree $k$ has parameters $(n,k,\lambda,\mu)$ such that two adjacent vertices have $\lambda$ common neighbors and two non-adjacent vertices has $\mu$ common neighbors. There exist two other important parameters for a strongly regular graph as follows.
\begin{equation*}
  \beta=\lambda - \mu, \Delta=(\lambda - \mu)^{2}+4(k-\mu).
\end{equation*}
Recall that a strongly regular graph with parameters $(n,k,\lambda,\mu)$ has three distinct eigenvalues $k$, $\frac{1}{2}(\beta \pm \sqrt{\Delta})$.

Let $G$ be a finite group and $S$ be an inverse-closed subset of $G$ not containing the identity element; we call $S$ the connection set. Then the Cayley graph $Cay(G,S)$ is the graph whose vertex set is $G$, where two vertices $a$ and $b$ are adjacent whenever $ab^{-1} \in S$. We note that the Cayley graph $Cay(G,S)$ is a regular graph of degree $|S|$ and it is connected if and only if the subgroup generated by the connection set $S$ is equal to $G$. The complement of a graph $\Gamma$ is denoted by $\Gamma^{c}$ and the identity element of a group by $e$. We also note that the complement of the Cayley graph $\Gamma=Cay(G,S)$ is $\Gamma^{c}=Cay(G,G \setminus (S \cup \{e\}))$.

A self-complementary graph is a graph that is isomorphic to its complement. It follows that if $\Gamma$ is a $k$-regular self-complementary graph with $n$ vertices, then $k=n-k-1$ and therefore $n=2k+1$. Moreover, $k$ must be an even number since the number of vertices is odd. This implies that the number of vertices of a $k$-regular self-complementary graph must be congruent to $1$ modulo $4$. Furthermore, it is well known that a self-complementary graph has diameter $d=2$ or $d=3$ because if a connected graph has diameter $d \geq 3$, then its complement has diameter $d \leq 3$.
\begin{prop} \label{Self-complementary regular}
The diameter of a self-complementary regular graph is two.
\end{prop}
\begin{proof}
Let $\Gamma$ be a self-complementary regular graph. Then it has $4t+1$ vertices and degree $2t$ for some natural numbers $t$. We calculate the diameter of $\Gamma^{c}$. Let $v$ and $w$ be two adjacent vertices of the graph $\Gamma$. If these two vertices have no common neighbor in the graph $\Gamma$, then it is easy to see that there is a unique vertex $u$ such that the vertices $v$ and $w$ are not adjacent to $u$ in the graph $\Gamma$ because this regular graph has degree $2t$ with $4t+1$ vertices. Therefore the distance between $v$ and $w$ is two in $\Gamma^{c}$. If these two vertices have at least one common neighbor in the graph $\Gamma$, then a similar argument works and there are more than one vertex such that the vertices $v$ and $w$ are not adjacent to these vertices in the graph $\Gamma$. Therefore the diameter of the graph $\Gamma^{c}$ is two. This implies that the diameter of the graph $\Gamma$ is two because it is isomorphic to its complement and this completes the proof.
\end{proof}
\begin{cor}
  Let $\Gamma$ be a self-complementary distance-regular graph. Then $\Gamma$ is a self-complementary strongly regular graph.
\end{cor}
Let $\Gamma$ be a self-complementary strongly-regular graph with parameters $(n=4t+1,k=2t,\lambda,\mu)$. Recall that there is a relation among the parameters of a strongly regular graph that is $(n-k-1)\mu=k(k-\lambda-1)$. It follows that $2t\mu=2t(2t-\lambda-1)$ and therefore $\mu+\lambda=2t-1$ in the graph $\Gamma$. On the other hand, the complement of a strongly regular graph with parameters $(n,k,\lambda,\mu)$ is a strongly regular graph with parameters $(n,n-k-1,n-2-2k+\mu,n-2k+\lambda)$. It follows that $\lambda-\mu=-1$ because the parameters of the self-complementary strongly regular graph $\Gamma$ and its complement are the same. This implies that $\lambda=t-1$ and $\mu=t$. Such a strongly regular graph is called a conference graph, i.e., the strongly regular graph with parameters $(4t+1,2t,t-1,t)$, and such parameters are called Payley type parameters. It follows that a self-complementary strongly regular graph is a conference graph. Furthermore, its distinct eigenvalues are $\{2t,\frac{-1\pm\sqrt{4t+1}}{2}\}$, $\beta=-1$, and $\Delta=n=4t+1$.

There exist two infinite families of self-complementary strongly regular Cayley graphs; Payley graphs and Peisert graphs.
Let $\mathbb{F}$ be a finite field of order $q=p^{r}$, where $p$ is a prime number and $q$ congruent to $1$ modulo $4$. Then the Paley graph $P_{q}$ is a Cayley graph $Cay(G,S)$ over the elementary abelian group $G$ with the connection set $S=\{x^{2} \mid x \in \mathbb{F}, x \neq 0\}$. If the multiplicative generator of the finite field $\mathbb{F}$ is denoted by $a$, where $p$ is a prime number congruent to $3$ modulo $4$ and $r$ is even, then the Peisert graph $P^{*}_{q}$ is a Cayley graph $Cay(G,S)$ over the elementary abelian group $G$ with the connection set $S=\{a^{i} \mid i \equiv 0,1 \mod 4\}$.
\subsection{The lexicographic product and self-complementary Cayley graphs}
Let $\Gamma_{1}$ and $\Gamma_{2}$ be two graphs with vertex set $V(\Gamma_{1})$ and $V(\Gamma_{1})$, respectively. Then the lexicographic product of the graph $\Gamma_{1}$ with $\Gamma_{2}$ which is denoted by $\Gamma_{1}[\Gamma_{2}]$, is a graph with vertex set $V(\Gamma_{1}) \times V(\Gamma_{2})$ such that two vertices $(a,b)$ and $(c,d)$, where $a,c \in V(\Gamma_{1})$ and $b,d \in V(\Gamma_{2})$, are adjacent whenever $a$ is adjacent to $c$ in $\Gamma_{1}$ or $a=c$ and $b$ is adjacent to $d$ in $\Gamma_{2}$. It is well known that the lexicographic product of a self-complementary graph $\Gamma_{1}$ with another self-complementary graph $\Gamma_{2}$ is a self-complementary graph. On the other hand, it is trivial to see that the lexicographic product of a Cayley graph $Cay(G_{1},S_{1})$ with another Cayley graph $Cay(G_{2},S_{2})$ is a Cayley graph $Cay(G_{1} \times G_{2},S)$, where $S=\{(a,g) \mid a \in S_{1}, g \in G_{2}\} \cup \{(e,s) \mid s \in S_{2}\}$. This implies that the lexicographic product of a self-complementary Cayley graph with another self-complementary Cayley graph is a self-complementary Cayley graph.
\subsection{Strongly regular Cayley graphs and Schur rings}
Let $G$ be a finite group of order $n$ and $D$ its subset of order $k$. Then the subset $D$ is called a $(n,k,\lambda,\mu)$-partial difference set in the group $G$ whenever every nonidentity element $g \in G$ can be expressible $\lambda$ and $\mu$ times as $d_{1}d_{2}^{-1}$ with the elements $d_{1}, d_{2} \in D$ depending on whether $g \in D$ or not. For more background, we refer to the excellent survey of partial difference sets by Ma \cite{MA}.

Let $G$ be a group and $R$ a commutative ring with identity. Then the group algebra $RG$ consists of the elements of form $\sum_{g \in G}a_{g}g$, where $a_{g} \in R$, with the following operations.
\begin{equation*}
\sum_{g \in G}a_{g}g+\sum_{g \in G}b_{g}g=\sum_{g \in G}(a_{g}+b_{g})g,
\end{equation*}
\begin{equation*}
(\sum_{g \in G}a_{g}g)(\sum_{h \in G}b_{h}h)=\sum_{g,h \in G}(a_{g}b_{h})gh,
\end{equation*}
and the scalar multiplication,
\begin{equation*}
c\sum_{g \in G}a_{g}g=\sum_{g \in G}(ca_{g})g.
\end{equation*}
Let $T$ be a subset of the group $G$. Then the element $\sum_{t \in T}t$, in this algebra, is denoted by $\overline{T}$.

Let $Cay(G,S)$ be a strongly regular Cayley graph with the parameters $(n,k,\lambda,\mu)$. Then the connection set $S$ is a partial difference set in the group $G$ and the following equation holds (cf. \cite[Proposition~1.1 and Theorem~1.3]{MA}).
\begin{equation} \label{Equation SRG}
  \overline{S}^{2}=\mu \overline{G}+(\lambda - \mu)\overline{S}+(k-\mu)e.
\end{equation}
Let $\{T_{0},T_{1},\ldots,T_{d}\}$, where $T_{0}=\{e\}$, be the partition of the group $G$ with disjoint inverse-closed subsets. Then the subalgebra $\mathcal{S}$ generated by $\alpha=\{\overline{T_{0}},\overline{T_{1}},\ldots,\overline{T_{d}}\}$ is called symmetric Schur ring over the ring $R$ whenever the $R$-module $\mathcal{S}$ is generated by $\alpha$, and in this case, $\alpha$ is called the simple basis of the Schur ring $\mathcal{S}$. We don't aim to deal with Schur rings in the general case in this paper.

Let $Cay(G,S)$ be a strongly regular Cayley graph with the parameters $(4t+1,2t,t-1,t)$ over the abelian group $G$ of order $4t+1$ for some natural numbers $t$. Then its complement is also a strongly regular Cayley graph with parameters $(4t+1,2t,t-1,t)$. It follows that the subalgebra $\mathcal{S}$ in the group algebra $\mathbb{Z}G$, where $\mathbb{Z}$ is the ring of integer numbers, generated by $\alpha=\{e,\overline{S},\overline{G \setminus (S \cup \{e\})}\}$ is indeed a Schur ring. To see this it is sufficient to consider the following equations (see Equation \ref{Equation SRG}).
\begin{equation*}
  \overline{S}^{2}=t \overline{G}-\overline{S}+te,
\end{equation*}
\begin{equation*}
 \overline{G \setminus (S \cup \{e\})}^{2}=t \overline{G}-\overline{G \setminus (S \cup \{e\})}+te.
\end{equation*}
Moreover,
\begin{equation*}
  \overline{G \setminus \{e\}}^{2}=(\overline{S}+ \overline{G \setminus (S \cup \{e\})})^{2}=\overline{S}^{2}+\overline{G \setminus (S \cup \{e\})}^{2}+2(\overline{S})(\overline{G \setminus (S \cup \{e\})}).
\end{equation*}
On the other hand,
\begin{equation*}
  \overline{G \setminus \{e\}}^{2}=(|G|-1)e+(|G|-2)(\overline{G \setminus \{e\}})=4t(e)+(4t-1)(\overline{G \setminus \{e\}}).
\end{equation*}
This implies that
\begin{equation}\label{Product coefficient}
 (\overline{S})(\overline{G \setminus (S \cup \{e\})})=t(\overline{G \setminus \{e\}})=t(\overline{S}+ \overline{G \setminus (S \cup \{e\})}).
\end{equation}
This implies that if $Cay(G,S)$ is a strongly regular Cayley graph with the parameters $(4t+1,2t,t-1,t)$, then every non-identity element $g$ in the group $G$ can be expressible $t$ times as the product of two distinct elements $x$ and $y$ in the connection set $S$ and $G \setminus S \cup \{e\}$, respectively.
\section{Self-complementary strongly regular Cayley graphs}
The Cayley graphs over cyclic groups are the same as circulant graphs and distance-regular circulant graphs have been classified in \cite{MP}. Therefore the Paley graphs on the prime number of vertices are the only self-complementary distance-regular circulant graphs. Let $G$ be the abelian group $G=\mathbb{Z}_{p^{2}} \times \mathbb{Z}_{p^{2}}$, where $p$ is a prime number. Davis \cite[Corollary~3.1]{Davis} constructed a Paley type partial difference set in the abelian group $G$ as follows. Let $G=\seq{(i,j)|i,j=0,1,\ldots,p^{2}-1}$ and $C$ be the set of the elements of order $p^{2}$ in the following subgroups of order $p^{2}$.
\begin{equation*}
  \{\seq{(1,1)},\seq{(1,2)},\ldots,\seq{(1,\frac{p(p-1)}{2})},\seq{(p,1)},\seq{(2p,1)},\ldots,\seq{(\frac{(p-1)p}{2},1)}\},
\end{equation*}
and $D$ be the set of all elements without the identity element in the following subgroups of order $p^{2}$.
\begin{equation*}
\{\seq{(1,0)},\seq{(0,1)},\seq{(1,\frac{p^{2}-p}{2}+1)},\ldots,\seq{(1,\frac{p^{2}+1}{2}-2)}\}.
\end{equation*}
Then $S=C \cup D$ is a Paley type partial difference set in the group $G$ which is inverse-closed. We checked it with GAP \cite{GAP} that the Cayley graph $Cay(G,S)$ is a self-complementary strongly regular graph for $p=3$. This shows that self-complementary strongly regular Cayley graphs over abelian groups are not on the elementary abelian groups and answers the second part of Question \ref{Question2}. But this construction is not self-complementary in general because is not self-complementary for $p=5$. Furthermore, the construction of partial difference sets with Paley parameters by Leung and Ma \cite{LM} does not give rise to self-complementary strongly regular Cayley graphs for the group $G=\mathbb{Z}_{27} \times \mathbb{Z}_{27}$.

We note that a self-complementary symmetric graph is isomorphic to a self-complementary strongly regular Cayley graph over an elementary abelian $p$-group, where $p$ is an odd prime number, and this family of self-complementary graphs has been classified in \cite{Peisert}; these are the Paley graphs, Peisert graphs, and the exceptional graph on $23^{2}$ vertices. As far as we know, there is no self-complementary strongly regular Cayley graph over the elementary abelian $p$-groups other than the mentioned graphs and therefore we conclude this paper with the following question.
\begin{Question}
  Is it true that the Paley graphs, Peisert graphs, and the exceptional graph on $23^{2}$ vertices are the only self-complementary strongly regular Cayley graphs over elementary abelian groups?
\end{Question}
\section*{Acknowledgements}
The author is grateful to the Research Council of Shahid Chamran University of Ahvaz for financial support (SCU.MM1401.29248).


\begin{thebibliography}{20}
\bibitem{AJMP} Arasu, Jungnickel, Ma and Pott, Note: Strongly regular Cayley graphs with $\lambda-\mu=-1$, J. Combin. Theorey Ser. A 67 (1994) 116--125.
\bibitem{BCN} A. E. Brouwer, A. M. Cohen and A. Neumaier, Distance-regular graphs, Springer-Verlag Berlin Heidelberg, New York, 1989.
\bibitem{Davis} J. A. Davis, Partial difference sets in $p$-groups, Arch. Math. 63 (1994) 103--110.
\bibitem{LM} K. H. Leung and S. L. Ma, Partial difference sets with Paley parameters, Bull. London Math. Soc. 27 (1995) 553--564.
\bibitem{MA} S. L. MA, A survey of partial difference sets, Des. Codes Cryptogr. 4 (1994) 221–-261.
\bibitem{MP} \v{S}. Miklavi\v{c}, P. Poto\v{c}nik, Distance-regular circulants, European J. Combin. 24 (2003) 777--784.
\bibitem{Peisert} W. Peisert, All Self-Complementary Symmetric Graphs, J. Algebra 240 (2001) 209--229.
\bibitem{Polhill} J. Polhill, Paley type partial difference sets in non $p$-groups, Des. Codes Cryptogr. 52 (2009) 163--169.
\bibitem{WLLW}
L. Wang, C.H. Li, Y. Liu, and C.X Wu, New constructions of self-complementary Cayley graphs, Electron. J. Combin. 24(3) (2017) \#P3.19.
\bibitem{Wang}
Z. Wang, Paley type partial difference sets in abelian groups, J. Comb. Des. 28(2) (2020) 149--152.
\bibitem{GAP} The GAP Group, \emph{GAP -- Groups, Algorithms, and Programming,   Version 4.12.1}; 2022,
\url{http://www.gap-system.org}.
\end{thebibliography}
\end{document}